\documentclass[10pt,reqno]{amsart}
\usepackage{amssymb}
\usepackage{amsmath}
\usepackage{amsfonts}
\usepackage[all]{xy}
\usepackage[english]{babel}
\usepackage{booktabs,multirow,graphicx}
\usepackage[notref,notcite,final]{showkeys}
\usepackage{enumerate}
\usepackage{verbatim}
\usepackage{graphicx}
\usepackage[fleqn]{mathtools}

\renewcommand{\phi}{\varphi}

\theoremstyle{plain}
\newtheorem{theorem}{Theorem}
\newtheorem*{theorem*}{Theorem}
\newtheorem{lemma}{Lemma}

\newtheorem{claim}{Claim}
\theoremstyle{remark}
\newtheorem{remark}{Remark}
\begin{document}

\title{Li-Yorke sensitivity does not imply Li-Yorke chaos}

\author{Jana Hant\'akov\'a}

\address{ Mathematical Institute, Silesian University, CZ-746 01 
Opava, Czech Republic}

\email{jana.hantakova@math.slu.cz }

\begin{abstract} We construct an infinite-dimensional compact metric space $X$, which is a closed subset of $\mathbb{S}\times\mathbb{H}$, where $\mathbb{S}$ is the unit circle and $\mathbb{H}$ is the Hilbert cube, and a skew-product map $F$ acting on $X$ such that $(X,F)$ is Li-Yorke sensitive but possesses at most countable scrambled sets. This disproves the conjecture of Akin and Kolyada that Li-Yorke sensitivity implies Li-Yorke chaos from the article [{\sc Akin E., Kolyada S.}, {\em Li-Yorke sensitivity\/}, Nonlinearity {\bf 16}, (2003), 1421--1433].

{\small {2000 {\it Mathematics Subject Classification.}}
Primary 37B05; 37D45, 54H20.
}
\end{abstract}

\maketitle
\section{Introduction}
Li-Yorke sensitivity and Li-Yorke chaos are well-known properties of dynamical systems, where by a dynamical system we mean a phase space $X$ endowed with an evolution map $T$. We require that the phase space $(X,d)$ is a compact metric space and the evolution map is a continuous surjective mapping $T:X \rightarrow X$. \\
The definition of Li-Yorke sensitivity is a combination of sensitivity and Li-Yorke chaos. The Li-Yorke chaos was introduced in 1975 by Li and Yorke in \cite{LY}. A dynamical system is \emph{Li-Yorke chaotic} if there is an uncountable scrambled set. A set $S$ is scrambled if any two distinct points $x,y\in S$ are proximal (i.e. trajectories of $x$ and $y$ are arbitrarily close for some times) but not asymptotic, that means
$$\liminf _{n\to\infty} d(T^n(x),T^n(y))=0\qquad\mbox{ and }\qquad \limsup _{n\to\infty} d(T^n(x),T^n(y))>0.$$

The initial idea of sensitivity goes back to Lorenz \cite{Lor}, but it was firstly used in topological dynamics by Auslander and Yorke in \cite{AY} and popularized later by Devaney in \cite{D}. A map $T$ is \emph{sensitive} if there is $\epsilon>0$ such that that for each $x\in X$ and each $\delta>0$ there is $y\in X$ with $d(x,y)<\delta$ and $n\in\mathbb{N}$ such that $d(T^n(x),T^n(y))>\epsilon$. By Huang and Ye in \cite{HY}, $T$ is sensitive if and only if there is $\epsilon>0$ with the property that any neighbourhood of any $x\in X$ contains a point $y$ such that trajectories of $x$ and $y$ are separated by $\epsilon$ for infinitely many times, that is
$$\limsup _{n\to\infty} d(T^n(x),T^n(y))>\epsilon.$$

Inspired by the above results, Akin and Kolyada introduced Li-Yorke sensitivity in \cite{AK}.  A map $T$ is \emph{Li-Yorke sensitive} if there is $\epsilon>0$ with the property that any neighbourhood of any $x\in X$ contains a point $y$ proximal to $x$, such that trajectories of $x$ and $y$ are separated by $\epsilon$ for infinitely many times. Thus,
$$\liminf _{n\to\infty} d(T^n(x),T^n(y))=0\qquad\mbox{ and }\qquad \limsup _{n\to\infty} d(T^n(x),T^n(y))>\epsilon.$$
Authors in \cite{AK} proved, among others, that weak mixing systems are Li-Yorke sensitive and stated five conjectures concerning Li-Yorke sensitivity. Three of them were disproved in \cite{M1} and \cite{M2}, one was confirmed recently in \cite{Y}. Only one problem remained open until today:\\

{\bf Question 1.} Are all Li-Yorke sensitive systems Li-Yorke chaotic?\\

This question was also included in the list of important open problems in the contemporary chaos theory in topological dynamics in \cite{YeL}.\\
We show that the answer is negative. We construct an infinite-dimensional compact metric space $X$, which is a closed subset of $\mathbb{S}\times\mathbb{H}$, where $\mathbb{S}$ is the unit circle and $\mathbb{H}$ is the Hilbert cube, and a skew-product map $F$, which is a combination of a rotation on $\mathbb{S}$ and a contraction on $\mathbb{H}$, such that $(X,F)$ is Li-Yorke sensitive but possesses at most countable scrambled sets. The mapping $F$ can be continuously extended to get a connected dynamical system with the same properties, see Remark 1. \\

We recall here some notations used throughout the paper. A pair of points $(x,y)$ in $X^2$ is \emph{asymptotic} if $\lim_{n\to\infty} d(T^n(x),T^n(y))=0$. A pair of points $(x,y)$ in $X^2$ is \emph{proximal} if $\liminf_{n\to\infty} d(T^n(x),T^n(y))=0$, if $(x,y)$ is not proximal then it is called \emph{distal}. A pair of points $(x,y)$ in $X^2$ is \emph{scrambled} if it is proximal but not asymptotic. A pair of points $(x,y)$ in $X^2$ is \emph{scrambled with modulus $\epsilon$} if it is proximal and $\limsup_{n\to\infty} d(T^n(x),T^n(y))\geq \epsilon$. A system $(X,T)$ is \emph{minimal} if every point $x\in X$ has a dense orbit $\{T^n(x)\}_{n=0}^{\infty}$. A system is \emph{transitive} if, for every pair of open, nonempty subsets $U,V\subset X$, there is a positive integer $n\in\mathbb{N}$ such that $U\cap T^{n}(V)\neq \emptyset.$ A system $(X,T)$ is \emph{weakly mixing} if the product system $(X\times X,T\times T)$ is transitive.\\
\section{Main result}
Here we state the main result and outline of its proof. Technical details of the proof can be found in a form of lemmas and claims in the last section.
\begin{theorem} There is a Li-Yorke sensitive dynamical system which is not chaotic in the sense of Li-Yorke.
\end{theorem}
\begin{proof}
Let $X_0$ be the unit circle $\mathbb{S}=\mathbb{R}/\mathbb{Z}$ equipped with the metric $d_0(x,y)=\min\{|x-y|,1-|x-y|\}$ and, for $i\geq 1$, $X_i=\mathbb{N}\cup\{\infty\}$ equipped with the metric  $d_i(x,y)=|\frac{1}{x}-\frac{1}{y}|$, where $\frac{1}{\infty}=0$. Then $\Pi_{i=0}^{\infty}X_i$ with the product topology is a compact space. The product topology is equivalent to the metric topology induced by the metric $D(x,y)=\sum_{i=0}^{\infty}\frac{d_i(x_i,y_i)}{2^i}.$ Let $Y=\{x\in\Pi_{i=0}^{\infty}X_i: \{x_i\}_{i=1}^{\infty}\mbox{ is nondecreasing}\}.$ $Y$ is a closed subset of  $\Pi_{i=0}^{\infty}X_i$ and therefore it is a compact metric space. Notice that, for $i\geq 1$, $X_i$ can be embed into the unit interval $[0,1]$ equipped with the natural topology, so $Y$ can be identify with a closed subset of $\mathbb{S}\times\mathbb{H}$, where $\mathbb{H}$ is the Hilbert cube.\\
Let $F:Y\rightarrow Y$ be a mapping defined for a point $x=(x_0,x_1,x_2,\ldots)$ in Y by $F(x)=(f_0(x),f_1(x),f_2(x),\ldots)$, where
\begin{align}f_0(x)&=(x_0+\sum_{i=1}^{\infty}\frac{1}{2^i}\cdot \frac{1}{x_i})\mod 1,\\
f_i(x)&=x_i+1,\qquad\mbox{ for }i\geq 1,
\end{align}
where $\infty+1=\infty.$
$F$ is a continuous mapping, since $f_i$ is continuous, for every $i\geq 0$.
First, we will show that $(F,Y)$ is Li-Yorke sensitive. It is enough to show that, for a given $x=(x_0,x_1,x_2,\ldots)\in Y$ and $U\in Nb_x$, there is $y=(y_0,y_1,y_2,\ldots)\in U$ such that 
$$\liminf _{n\to\infty} D(F^n(x),F^n(y))=0\qquad\mbox{ and }\qquad \limsup _{n\to\infty} D(F^n(x),F^n(y))\geq \frac 1 2.$$

\emph{ I. $\{x_i\}_{i=1}^{\infty}$  is a nondecreasing sequence containing at least one $\infty$}\\
Since  $\{x_i\}_{i=1}^{\infty}$ is nondecreasing, there is $M\in\mathbb{N}$ such that $x_i$ is finite, for $i<M$, and $x_i=\infty$, for $i\geq M$. The neigbourhood $U$ is defined by  $U=V\cap Y$, where $V$ is a neigbourhood of $x$ in $\Pi_{i=0}^{\infty}X_i$. Let $V=V_0\times V_1\times V_2\times\ldots$, where $V_i$ is a neigbourhood of $x_i$ such that  $V_i=X_i$ for all but finitely many $i\geq 0$. Let $K\in \mathbb{N}$, sufficiently large to satisfy $K\in V_i$, for $i\geq M$, and simultaneously $K\geq x_{M-1}$. We define the point $y$ as follows:
\begin{align}\label{y1}
y_i=&x_i, \mbox{ for }0\leq i<M,\\
\label{y11}y_i=&K, \mbox{ for }i\geq M.
\end{align}
It is easy to see that $y$ belongs to $U$. By Claim 1, $(x,y)$ is scrambled with modulus $\frac 1 2$.\\

\emph{II. $\{x_i\}_{i=1}^{\infty}$ is a nondecreasing sequence of finite numbers}\\
The neigbourhood $U$ is defined by $U=V\cap Y$, where $V$ is a neigbourhood of $x$ in $\Pi_{i=0}^{\infty}X_i$. Without loss of generality, suppose $V=V_0\times V_1\times V_2\times\ldots$, where $V_0=(x_0-\delta,x_0+\delta)$, for some $\delta>0$, and, for $i>0$, $V_i$ is a neigbourhood of $x_i$ such that  $V_i=X_i$ for all but finitely many $i$. Let $M\in \mathbb{N}$ such that $2^{-M}<\delta$ and simultaneously $V_i=X_i$, for $i\geq M$. We define the point $y$ as follows:
\begin{align}
\label{y2}y_0=&\big(x_0+1-\sum_{i=1}^{\infty}\frac{1}{2^{i+M}}\delta_i\big)\mod 1,\\
\label{y22}y_i=&x_i, \mbox{ for }0<i<M,\\
\label{y222}y_i=&\infty, \mbox{ for }i\geq M,
\end{align}
where
$$\delta_i =
\left\{
	\begin{array}{ll}
		0,  \quad&\mbox{if } x_{i+M}=x_M, \\
		\big(\sum_{j=0}^{x_{i+M}-x_M-1}\frac{1}{x_M+j}\big)\mod 1,\quad & \mbox{otherwise.}
	\end{array}
\right.$$

It is easy to see that $\sum_{i=1}^{\infty}\frac{1}{2^{i+M}}\delta_i\leq 2^{-M}$ and $y$ belongs to $U$.  By Claim 2, $(x,y)$ is scrambled with modulus $\frac 1 2$.\\

Notice that, in both cases, one point of the pair $(x,y)$ has $\infty$ coordinates while the other has all coordinates finite. By Claim 3, if $x_i$ and $y_i$ are finite, for all $i\geq 1$, then $\lim_{n\to\infty}D(F^n(x),F^n(y))$ exists and $(x,y)$ is not a scrambled pair.
Therefore in each scrambled set $S\subset Y$, there is at most one $z\in S$ such that $z_i$ is finite, for $i\geq 1$. We finish our proof by finding an injection between $S\setminus\{z\}$ and $\mathbb{N}$.\\
Let $l_x=\min\{i: x_i=\infty\}$. Then the mapping $\iota:S\setminus\{z\}\rightarrow \mathbb{N}$ defined by $\iota(x)=\l_x$ is injective. We proceed by assuming the opposite. Let $x\neq y$ in $S\setminus\{z\}$ such that $l=l_x=l_y$. Since $\{x_i\}_{i=1}^{\infty}$ and  $\{y_i\}_{i=1}^{\infty}$ are nondecreasing,
\begin{equation}\label{y3} x_i<\infty \wedge y_i<\infty, \quad\mbox{ for }0<i<l,\qquad x_i=y_i=\infty, \quad\mbox{ for }i\geq l. \end{equation}
By Claim 4, $\lim_{n\to\infty}D(F^n(x),F^n(y))$ exists which is in contradiction with $(x,y)$ being a scrambled pair.
\end{proof}
\begin{remark} The mapping $F$ can be continuously extended to get a \emph{connected} dynamical system with the same properties. Let $X_0$ be the unit circle $\mathbb{S}=\mathbb{R}/\mathbb{Z}$ equipped with the metric $d_0(x,y)=\min\{|x-y|,1-|x-y|\}$ and, for $i\geq 1$, $X_i$ be the unit interval $[0,1]$ equipped with the natural topology. Then $\Pi_{i=0}^{\infty}X_i$ equipped with the product topology is $\mathbb{S}\times \mathbb{H}$, where $\mathbb{H}$ is the Hilbert cube. The product topology is equivalent to the metric topology induced by the metric $D(x,y)=\sum_{i=0}^{\infty}\frac{d_i(x_i,y_i)}{2^i}.$ Let $Z=\{x\in\mathbb{S}\times \mathbb{H}: \{x_i\}_{i=1}^{\infty}\mbox{ is nonincreasing}\}.$ $Z$ is closed and pathwise connected subset of  $\mathbb{S}\times \mathbb{H}$.\\
Let $x=(x_0,x_1,x_2,\ldots)$ be a point in $Z$. We will express every $x_i\in X_i\setminus\{0\}=(0,1]$, for $i\geq1$, as $x_i=\frac{1}{k_i}+t_i\cdot|\frac{1}{k_i-1}-\frac{1}{k_i}|,$ where $t_i\in(0,1]$ and $k_i\in\mathbb{N}\setminus\{1\}$.
Let $G:Z\rightarrow Z$ be a mapping defined by $G(x)=(g_0(x)g_1(x),g_2(x),\ldots)$, where
$$g_0(x)=(x_0+\sum_{i=1}^{\infty}\frac{1}{2^i}\cdot {x_i})\mod 1,$$
and for $i\geq 1,$
$$g_i(x)=
\left\{
\begin{array}{ll}
0,  \quad&\mbox{if } x_{i}=0,\\
\frac{1}{k_i+1}+t_i\cdot|\frac{1}{k_i}-\frac{1}{k_i+1}|,\quad &\mbox{otherwise.}
\end{array}
\right.$$
Then $G$ is a continuous extension of $F$ and $(G,Z)$ is a Li-Yorke sensitive but not Li-Yorke chaotic system. \end{remark}
\begin{remark}
The mapping $F$ is not minimal (it is even not transitive). In case of minimal maps, we have still an open question:\\

{\bf Question 2.} Are all Li-Yorke sensitive minimal systems Li-Yorke chaotic?
\end{remark}
\begin{remark}
$Y$ is an infinite-dimensional space. We can examine the relation between Li-Yorke sensitivity and Li-Yorke chaos for low-dimensional dynamical systems. It is known that in case of graph mappings (in particular, interval mappings) Li-Yorke sensitivity implies Li-Yorke chaos, since, for graph mappings, the existence of a single scrambled pair implies the existence of an uncountable scrambled set. But this is not true for other classes of dynamical systems - shifts, maps on dendrites, triangular maps of the square.\\

{\bf Question 3.} Are all Li-Yorke sensitive shifts/maps on dendrites/triangular maps of the square Li-Yorke chaotic?\\
\end{remark}
\section{Proofs}
\begin{lemma}
Let $p\in\mathbb{N}$,$m\in\mathbb{N}$ and $\epsilon>0.$ There are sequences $\{v_n\}_{n=1}^{\infty}$ and $\{u_n\}_{n=1} ^{\infty}$ such that
\begin{equation}\label{v}\mbox{for every }n\in\mathbb{N}, \qquad \Big(\frac{1}{2^p}\sum_{j=0}^{v_n-1}\frac{1}{m+j}\Big) \mod 1<\epsilon,
\end{equation}
and
\begin{equation}\label{u}\mbox{for every }n\in\mathbb{N},\qquad \Big|\Big(\frac{1}{2^p}\sum_{j=0}^{u_n-1}\frac{1}{m+j}\Big) \mod 1-\frac 1 2\Big|<\epsilon.
\end{equation}
\end{lemma}
Lemma 1  follows by the simple fact that the harmonic series is divergent while its increment tends to 0. Therefore the $n$th partial sum of harmonic series modulo 1 is $\epsilon$-close to any number from $[0,1)$ for infinitely many $n$.
\begin{lemma}
For any $i\in\mathbb{N}$, let $\{\delta^n_i\}_{n=1}^{\infty}$ be a sequence of positive numbers not greater than 1, such that $\lim_{n\to\infty}\delta_i^n=0$. Then $$\lim_{n\to\infty}\sum_{i=1}^{\infty}\frac{1}{2^i}\delta_i^n=0.$$
\end{lemma}
\begin{proof}
For every $\epsilon>0$, there is $k\in\mathbb{N}$ such that $\epsilon>2^{-k+1}$. Since  $\lim_{n\to\infty}\delta_i^n=0$, for every $i\in\mathbb{N}$, there is $N\in\mathbb{N}$ such that, for $n\geq N$ and $i\leq k$, $\delta_i^n<2^{-k}$. We can estimate, for $n\geq N$,
$$\sum_{i=1}^{\infty}\frac{1}{2^i}\delta_i^n<\sum_{i=1}^{k}\frac{1}{2^i}\delta_i^n+\sum_{i=k+1}^{\infty}\frac{1}{2^i}<(1-2^{-k})\cdot 2^{-k}+2^{-k}<\epsilon.$$
\end{proof}
\begin{lemma}
Let $k\in\mathbb{N}$, $r\in\mathbb{N}$. Then $$\lim_{n\to\infty}\sum_{j=0}^{n-1}\Big(\frac{1}{k+j}-\frac{1}{k+r+j}\Big)=\sum_{j=0}^{r-1}\frac{1}{k+j}.$$
\end{lemma}
\begin{proof} For sufficiently large $n$,
$$\sum_{j=0}^{n-1}\Big(\frac{1}{k+j}-\frac{1}{k+r+j}\Big)=\sum_{j=0}^{r-1}\frac{1}{k+j}-\sum_{j=n-r}^{n-1}\frac{1}{k+r+j},$$
where the second term on the right side tends to 0 for $n\to\infty.$
\end{proof}
\begin{claim} $x$ and $y$ defined in (\ref{y1}) and (\ref{y11}) is a scrambled pair with modulus $\frac 1 2$. \end{claim}
\begin{proof}
Denote the $i$th coordinate of $F^n(x)$ by $x_i^n$. First members of the sequences $\{x^n_0\}_{n=1}^{\infty}$ and $\{y^n_0\}_{n=1}^{\infty}$ are:
$$x_0\mapsto \big(x_0+\sum_{i=1}^{M-1}\frac{1}{2^i}\frac{1}{x_i}\big)\mod 1\mapsto \big(x_0+\sum_{i=1}^{M-1}\frac{1}{2^i}\frac{1}{x_i}+\sum_{i=1}^{M-1}\frac{1}{2^i}\frac{1}{x_i+1}\big)\mod 1 \ldots,$$
$$y_0=x_0\mapsto \big(x_0+\sum_{i=1}^{M-1}\frac{1}{2^i}\frac{1}{x_i}+\sum_{i=M}^{\infty}\frac{1}{2^i}\frac{1}{K}\big)\mod 1\mapsto \big(x_0+\sum_{i=1}^{M-1}\frac{1}{2^i}\frac{1}{x_i}\sum_{i=M}^{\infty}\frac{1}{2^i}\frac{1}{K}+\sum_{i=1}^{M-1}\frac{1}{2^i}\frac{1}{x_i+1}+\sum_{i=M}^{\infty}\frac{1}{2^i}\frac{1}{K+1}\big)\mod 1 \ldots.$$
The following equations are with modulus 1 whenever necessary. Since $d_0(x_0^n,y_0^n)\leq |x^n_0-y^n_0|$, where
\begin{equation}\label{abs}|x^n_0-y^n_0|=\sum_{i=M}^{\infty}\sum_{j=0}^{n-1}\frac{1}{2^i}\frac{1}{K+j}=2^{-M+1}\cdot\sum_{j=0}^{n-1}\frac{1}{K+j},\end{equation}
and
$$ d_i(x_i^n,y_i^n)=0, \quad\mbox{ for }0<i<M,\qquad d_i(x_i^n,y_i^n)=\frac{1}{K+n}, \quad\mbox{ for }i\geq M,$$
we can estimate
$$D(F^n(x),F^n(y))\leq2^{-M+1}\cdot\sum_{j=0}^{n-1}\frac{1}{K+j}+\sum_{i=M}^{\infty}\frac{1}{2^{i}}\frac{1}{K+n}.$$
Let $\epsilon>0.$ By (\ref{v}) in Lemma 1, there is $\{v_n\}_{n=1}^{\infty}$ such that $\big(2^{-M+1}\cdot\sum_{j=0}^{v_n-1}\frac{1}{K+j}\big)\mod 1<\epsilon$, for $n\geq 1$. By Lemma 2 and since $\lim_{n\to\infty}\frac{1}{K+n}=0$, it holds, for sufficiently large $v$, $\sum_{i=M}^{\infty}\frac{1}{2^{i}}\frac{1}{K+{v_n}}<\epsilon.$
Therefore $\lim_{v\to\infty}D(F^{v_n}(x),F^{v_n}(y))<2\epsilon$ and $\liminf_{n\to \infty}D(F^n(x),F^n(y))=0.$
Similarly by (\ref{u}) in Lemma 1, there is $\{u_n\}_{n=1}^{\infty}$ such that $\Big|\Big(2^{-M+1}\sum_{j=0}^{u_n-1}\frac{1}{K+j}\Big) \mod 1-\frac 1 2\Big|<\epsilon,$ for sufficiently large $n$. Therefore, by (\ref{abs}), $d_0(x_0^{u_n},y_0^{u_n})>\frac 1 2 -\epsilon$ and $\limsup_{n\to \infty}D(F^n(x),F^n(y))\geq\frac 1 2.$
\end{proof}
\begin{claim}  $x$ and $y$ defined in (\ref{y2}), (\ref{y22}) and (\ref{y222}) is a scrambled pair with modulus $\frac 1 2$. \end{claim}
\begin{proof}
Denote the $i$th coordinate of $F^n(x)$ by $x_i^n$. First members of the sequences $\{x^n_0\}_{n=1}^{\infty}$ and $\{y^n_0\}_{n=1}^{\infty}$ are:
$$x_0\mapsto \big(x_0+\sum_{i=1}^{\infty}\frac{1}{2^i}\frac{1}{x_i}\big)\mod 1\mapsto \big(x_0+\sum_{i=1}^{\infty}\frac{1}{2^i}\frac{1}{x_i}+\sum_{i=1}^{\infty}\frac{1}{2^i}\frac{1}{x_i+1}\big)\mod 1 \ldots,$$

\begin{multline*}y_0=\big(x_0+1-\sum_{i=1}^{\infty}\frac{1}{2^{i+M}}\delta_i\big)\mod 1\mapsto \big(x_0+1-\sum_{i=1}^{\infty}\frac{1}{2^{i+M}}\delta_i+\sum_{i=1}^{M-1}\frac{1}{2^i}\frac{1}{x_i}\big)\mod 1\mapsto \\\mapsto\big(x_0+1-\sum_{i=1}^{\infty}\frac{1}{2^{i+M}}\delta_i+\sum_{i=1}^{M-1}\frac{1}{2^i}\frac{1}{x_i}+\sum_{i=1}^{M-1}\frac{1}{2^i}\frac{1}{x_i+1}\big)\mod 1 \ldots .
\end{multline*}
Notice that for sufficiently large $n$, 
\begin{equation}\label{gamma}\delta_i+\sum_{j=0}^{n-1}\frac{1}{x_{i+M}+j}=\sum_{j=0}^{n-1}\frac{1}{x_M+j}+\gamma_i^n,\end{equation}
where 
$$\gamma^n_i =
\left\{
	\begin{array}{ll}
		0,  \quad&\mbox{if } x_{i+M}=x_M, \\
		\big(\sum_{j=x_M}^{x_{i+M}-1}\frac{1}{n+j}\big)\mod 1,\quad & \mbox{otherwise.}
	\end{array}
\right.$$
The following equations are with modulus 1 whenever necessary. Since
\begin{multline*}
d_0(x_0^n,y_0^n)\leq |x^n_0-y^n_0|=\sum_{i=1}^{\infty}\frac{1}{2^{i+M}}\delta_i+\sum_{i=M}^{\infty}\sum_{j=0}^{n-1}\frac{1}{2^i}\frac{1}{x_i+j}=\sum_{i=1}^{\infty}\frac{1}{2^{i+M}}\big(\delta_i+\sum_{j=0}^{n-1}\frac{1}{x_{i+M}+j}\big)+\sum_{j=0}^{n-1}\frac{1}{2^M}\frac{1}{x_M+j}\\
\overset{(\ref{gamma})}{=}\sum_{i=0}^{\infty}\frac{1}{2^{i+M}}\big(\sum_{j=0}^{n-1}\frac{1}{x_{M}+j}+\gamma_i^n\big)+\sum_{j=0}^{n-1}\frac{1}{2^M}\frac{1}{x_M+j}=2^{-M+1}\cdot\sum_{j=0}^{n-1}\frac{1}{x_M+j}+\sum_{i=0}^{\infty}\frac{1}{2^{i+M}}\gamma_i^n,
\end{multline*}
and
$$ d_i(x_i^n,y_i^n)=0, \quad\mbox{ for }0<i<M,\qquad d_i(x_i^n,y_i^n)=\frac{1}{x_i+n}, \quad\mbox{ for }i\geq M,$$
we can estimate
$$D(F^n(x),F^n(y))\leq2^{-M+1}\cdot\sum_{j=0}^{n-1}\frac{1}{x_M+j}+\sum_{i=0}^{\infty}\frac{1}{2^{i+M}}\gamma_i^n+\sum_{i=M}^{\infty}\frac{1}{2^{i}}\frac{1}{x_i+n}.$$
Let $\epsilon>0.$ By (\ref{v}) in Lemma 1, there is $\{v_n\}_{n=1}^{\infty}$ such that $\big(2^{-M+1}\cdot\sum_{j=0}^{v_n-1}\frac{1}{x_M+j}\big)\mod 1<\epsilon$, for $n\geq 1$. By Lemma 2 and since $\lim_{n\to\infty}\gamma_i^n=0$ and $\lim_{n\to\infty}\frac{1}{x_i+n}=0$, for $i\geq 1$, it holds, for sufficiently large $n$, $$\sum_{i=0}^{\infty}\frac{1}{2^{i+M}}\gamma_i^{v_n}<\epsilon\qquad\mbox{ and }\qquad\sum_{i=M}^{\infty}\frac{1}{2^{i}}\frac{1}{x_i+{v_n}}<\epsilon.$$
Therefore $\lim_{v\to\infty}D(F^{v_n}(x),F^{v_n}(y))<3\epsilon$ and $\liminf_{n\to \infty}D(F^n(x),F^n(y))=0.$
Similarly by (\ref{u}) in Lemma 1, there is $\{u_n\}_{n=1}^{\infty}$ such that $$\Big|\Big(2^{-M+1}\sum_{j=0}^{u_n-1}\frac{1}{x_M+j}\Big) \mod 1+\sum_{i=0}^{\infty}\frac{1}{2^{i+M}}\gamma_i^{u_n}-\frac 1 2\Big|<2\epsilon,$$ for sufficiently large $n$. Therefore $d_0(x_0^{u_n},y_0^{u_n})>\frac 1 2 -2\epsilon$ and $\limsup_{n\to \infty}D(F^n(x),F^n(y))\geq\frac 1 2.$\\
\end{proof}
\begin{claim} If $x_i$ and $y_i$ are finite, for all $i\geq 1$, then $\lim_{n\to\infty}D(F^n(x),F^n(y))$ exists and $(x,y)$ is not a scrambled pair.\end{claim}
\begin{proof}
Let $r_i=|x_i-y_i|.$ The following equations are with modulus 1 whenever necessary. First, observe by Lemma 3,
\begin{equation}\label{x}\lim_{n\to\infty}\sum_{\substack{i=1 \\ x_i\leq y_i}}^{\infty}\frac{1}{2^i}\sum_{j=0}^{n-1}\big(\frac{1}{x_i+j}-\frac{1}{y_i+j}\big)=\lim_{n\to\infty}\sum_{\substack{i=1 \\ x_i\leq y_i}}^{\infty}\frac{1}{2^i}\sum_{j=0}^{n-1}\big(\frac{1}{x_i+j}-\frac{1}{x_i+r_i+j}\big)=\sum_{\substack{i=1 \\ x_i\leq y_i}}^{\infty}\frac{1}{2^i}\sum_{j=0}^{r_i-1}\frac{1}{x_i+j}\end{equation}
and similarly,
\begin{equation}\label{y}\lim_{n\to\infty}\sum_{\substack{i=1 \\ x_i>y_i}}^{\infty}\frac{1}{2^i}\sum_{j=0}^{n-1}\big(\frac{1}{y_i+j}-\frac{1}{x_i+j}\big)=\sum_{\substack{i=1 \\ x_i>y_i}}^{\infty}\frac{1}{2^i}\sum_{j=0}^{r_i-1}\frac{1}{y_i+j}.\end{equation}
Since
\begin{multline}\label{cele}|x^n_0-y^n_0|=\big|x_0-y_0+\sum_{i=1}^{\infty}\frac{1}{2^i}\sum_{j=0}^{n-1}\big(\frac{1}{x_i+j}-\frac{1}{y_i+j}\big)\big|=\\\big|x_0-y_0+\sum_{\substack{i=1 \\ x_i\leq y_i}}^{\infty}\frac{1}{2^i}\sum_{j=0}^{n-1}\big(\frac{1}{x_i+j}-\frac{1}{y_i+j}\big)-\sum_{\substack{i=1 \\ x_i>y_i}}^{\infty}\frac{1}{2^i}\sum_{j=0}^{n-1}\big(\frac{1}{y_i+j}-\frac{1}{x_i+j}\big)\big|,\end{multline}
and
$$d_i(x_i^n,y_i^n)=\big|\frac{1}{x_i+n}-\frac{1}{y_i+n}\big|, \quad\mbox{ for }i\geq 1,$$
it follows by Lemma 2 and by $\lim_{n\to\infty}\big|\frac{1}{x_i+n}-\frac{1}{y_i+n}\big|=0$,
$$\lim_{n\to\infty}D(F^n(x),F^n(y))=\lim_{n\to\infty}d_0(x_0^n,y_0^n)+\lim_{n\to\infty}\sum_{i=1}^{\infty}\frac{1}{2^i}\big|\frac{1}{x_i+n}-\frac{1}{y_i+n}\big|=\min\{\lim_{n\to\infty}|x_0^n-y_0^n|,1-\lim_{n\to\infty}|x_0^n-y_0^n|\},$$
where by equations (\ref{x}), (\ref{y}), (\ref{cele}),
$$\lim_{n\to\infty}|x_0^n-y_0^n|=\big|x_0-y_0+\sum_{\substack{i=1 \\ x_i\leq y_i}}^{\infty}\frac{1}{2^i}\sum_{j=0}^{r_i-1}\frac{1}{x_i+j}-\sum_{\substack{i=1 \\ x_i>y_i}}^{\infty}\frac{1}{2^i}\sum_{j=0}^{r_i-1}\frac{1}{y_i+j}\big |.$$
 Therefore $\lim_{n\to\infty}D(F^n(x),F^n(y))$ exists which is in contradiction with $(x,y)$ being a scrambled pair.
\end{proof}
\begin{claim}  For $x$ and $y$ defined in (\ref{y3}), $\lim_{n\to\infty}D(F^n(x),F^n(y))$ exists and $(x,y)$ is not a scrambled pair.. \end{claim}
\begin{proof}
Let $r_i=|x_i-y_i|$, for $0<i<l$. By similar calculation as in Claim 3, 
$$\lim_{n\to\infty}D(F^n(x),F^n(y))=\lim_{n\to\infty}d_0(x_0^n,y_0^n)+\lim_{n\to\infty}\sum_{i=1}^{l-1}\frac{1}{2^i}\big|\frac{1}{x_i+n}-\frac{1}{y_i+n}\big|=\min\{\lim_{n\to\infty}|x_0^n-y_0^n|,1-\lim_{n\to\infty}|x_0^n-y_0^n|\},$$
where
$$\lim_{n\to\infty}|x_0^n-y_0^n|=\big|x_0-y_0+\sum_{\substack{i=1 \\ x_i\leq y_i}}^{l-1}\frac{1}{2^i}\sum_{j=0}^{r_i-1}\frac{1}{x_i+j}-\sum_{\substack{i=1 \\ x_i>y_i}}^{l-1}\frac{1}{2^i}\sum_{j=0}^{r_i-1}\frac{1}{y_i+j}\big |.$$
 Therefore $\lim_{n\to\infty}D(F^n(x),F^n(y))$ exists which is in contradiction with $(x,y)$ being a scrambled pair.
\end{proof}
\paragraph{\bf Acknowledgment}
The author thanks to professor Jaroslav Sm\' ital for stimulating discussions and useful suggestions. Research was funded by institutional support for the development of research organization (I\v C 47813059).

\end{document}